\documentclass[10pt]{amsart}

\usepackage[margin=0.7in]{geometry}

\usepackage{amsmath, amssymb, amsthm, amssymb}
\usepackage{graphpap}
\usepackage{amsfonts}
\usepackage{epsf}
\usepackage{epsfig}
\usepackage{latexsym}
\usepackage{graphicx}
\usepackage{amsbsy}
\usepackage{euscript}           
\usepackage[frame,ps,matrix,arrow,curve,rotate]{xy}
\usepackage{verbatim}
\usepackage{graphics}
\usepackage{enumerate}
\usepackage{color}
\usepackage{tikz}
\usepackage{hyperref}
\usepackage{changepage}

\newtheorem{theorem}{Theorem}[section]
\newtheorem{lemma}[theorem]{Lemma}

\newtheorem{corollary}[theorem]{Corollary}
\newtheorem{proposition}[theorem]{Proposition}

\newtheorem{claim}{Claim}[section]

\theoremstyle{definition}
\newtheorem{definition}[theorem]{Definition}

\theoremstyle{remark}
\newtheorem{remark}[theorem]{Remark}

\newcommand{\mc}[1]{\mathcal{#1}}

\DeclareMathOperator{\At}{At}

\makeatletter
\newcommand\mathcircled[1]{%
	\mathpalette\@mathcircled{#1}%
}
\newcommand\@mathcircled[2]{%
	\tikz[baseline=(math.base)] \node[draw,circle,inner sep=1pt] (math) {$\m@th#1#2$};%
}
\makeatother

\begin{document}
	
\title{Computable Stone spaces}

\author{Nikolay Bazhenov}
\address{Sobolev Institute of Mathematics, Novosibirsk, Russia}
\address{Novosibirsk State University, Novosibirsk, Russia}
\email{bazhenov@math.nsc.ru} 

\author{Matthew Harrison-Trainor}
\address{Department of Mathematics, University of Michigan, MI, USA}
\email{matthhar@umich.edu}

\author{Alexander Melnikov}
\address{School of Mathematics and Statistics, Victoria University of Wellington, New Zealand \and Sobolev Institute of Mathematics, Novosibirsk, Russia}
\email{alexander.g.melnikov@gmail.com} 

\thanks{Melnikov and Harrison-Trainor were partially supported by RDF - VUW1902,  Royal Society Te Ap\={a}rangi.}

\thanks{The work of Bazhenov and Melnikov is supported by the Mathematical Center in Akademgorodok under agreement No. 075-15-2022-281 with the Ministry of Science and Higher Education of the Russian Federation.}

\date{\today}

\begin{abstract}
We investigate computable metrizability of Polish spaces up to homeomorphism. In this paper we focus on Stone spaces. 
We use Stone duality to construct the first known example of a computable topological Polish space not homeomorphic to any computably metrized space.
In fact, in our proof we construct a right-c.e.~metrized Stone space which is not homeomorphic to any computably metrized  space.
Then we  introduce a new notion of effective categoricity for effectively compact spaces and prove that effectively categorical Stone spaces are exactly the duals of computably categorical Boolean algebras. 
Finally,  we prove that, for a Stone space $X$,
the Banach space $C(X;\mathbb{R})$  has a computable presentation if, and only if,  $X$ is \emph{homeomorphic} to a computably metrized space.
This gives an unexpected positive partial answer to a question recently posed by McNicholl.
\end{abstract}

\keywords{Stone space, computable topological space, computable categoricity, computable Polish space}

\subjclass{03D78}

\maketitle

\tableofcontents

\section{Introduction}

In this paper we use Boolean algebras and Stone spaces to prove new results in computable topology and computable Banach space theory.
The paper contributes to a fast-developing subject in computable analysis which is focused on applying effective algebraic techniques to study separable spaces and Polish groups.
This program was initiated in \cite{MelIso}; we cite \cite{Clanin.McNicholl.Stull.2019,McNellp,leb} for several recent results into this direction, and we also cite ~\cite{MDsurvey} for a detailed exposition of this approach.
The main objects of study in this theory are computable separable Banach spaces and computable Polish spaces and groups.
The main idea behind this new approach to separable spaces is that computable algebra can be viewed as a special case of separable metric space theory. In particular, with some effort one can sometimes extend results and techniques from computable discrete algebra to separable spaces. To a researcher interested in constructive aspect of mathematics, such investigations give a fine grained ``constructive'' analysis of proofs and processes in separable structures. To those more interested in classical mathematics, results of this sort provide formal estimates for the complexity of the classification problem in familiar classes of separable structures and spaces.\footnote{Indeed, as explained in \cite{MDsurvey, DoMo, GonKni},  computable results tend to be relativizable to an arbitrary oracle, and thus can be systematically used to measure the complexity of the classification problem for not necessarily computable structures. However, in this particular paper applications to classification problems will not be considered.}

Unfortunately, generalizing proofs and ideas of computable algebra to discrete spaces  is far from being routine or even systematic. Rather, such results  tend to rely on  intricate manipulations with presentations  and exploit results from multiple different areas of mathematics. For example, \cite{MelIso}  uses Pontryagin duality from abstract harmonic analysis, pregeometries from model theory, and a theorem of Dobrica~\cite{feebledobrica} from computable group theory. Further instances of this phenomenon can be found in \cite{HKSel-arxiv,  HTMN-ta, lupini} which blend advanced priority techniques, definability, and homological algebra with new methods specific to the subject. Thus, one of the main goals of the emerging theory is to bring some  order and system into the chaos of this tricky mathematics. For that, we need more general ``standard'' results and techniques one can systematically use beyond one specific application.

We believe that the present article partially fulfils this goal. More specifically,  we develop a certain technique that allows to transfer computability-theoretic results about countable Boolean algebras to results about separable spaces. This idea is, of course, not entirely new. Classically, one uses Stone duality to study completely disconnected compact spaces. The recent papers \cite{HKSel-arxiv,  HTMN-ta} apply two different computable versions of Stone duality to solve an open problem in computable topology.
We also cite~\cite{Tran:thesis}.
In the present paper we use Stone duality to prove three theorems about computable separable spaces. Even though we do need novel ideas different from those in the aforementioned works 
\cite{HKSel-arxiv,  HTMN-ta}, there is a certain systematic \emph{approximate definability analysis} of totally disconnected subspaces which unites all our results. This relationship is more of a technical one, but when the reader sees the proofs they will likely agree that our methods are systematic rather than based on a collection of tricks\footnote{While the paper was under review, our methods have found several further applications. For instance, various further effective dualities have been established in~\cite{lupini,EffedSurvey,tdlc}.}.

We are ready to discuss our results. Before we do so, we note that one of our main results (Theorem~\ref{thm1}) is concerned with Banach spaces rather than Polish spaces; nonetheless, it will almost immediately become clear how Stone spaces can help us in the study of computable  Banach spaces.

\

Our first main result answers a fundamental question in the foundations of computable topology.
One of the first tasks in any emerging theory is to establish the equivalence (or non-equivalence) of some of the most basic definitions and assumptions
which lie at the foundations of the theory. Point-set topology is notorious for its zoo of various notions of regularity of spaces, the most fundamental of which are known to be non-equivalent via relatively straightforward but clever counterexamples.
In stark contrast, \emph{computable} topology seems to be essentially completely missing the proofs that many of its computability-theoretic notions are (non-)equivalent.
This is partially explained by the fact that proving (non-)equivalence of such notions presents a significant challenge. For instance, it takes much effort even to prove that there exists a $
\Delta^0_2$-metrized Polish space not homeomorphic to a computably metrized one\footnote{All terms used in the introduction will be clarified in the preliminaries.}; see \cite{HKSel-arxiv, HTMN-ta} for three substantially different proofs all of which are  non-trivial. Another example is the recent construction of a computably metrized compact space not homeomorphic to any  effectively compact space suggested in \cite{HKSel-arxiv}.
Recall that a computably metrized space is \emph{effectively compact} if there is a computable enumeration of all of its finite covers by basic open balls of radius $2^{-n}$, uniformly in $n$. The recent paper~\cite{lupini} uses topological group theory and homological algebra to produce  an example of a computably metrized connected compact Polish \emph{group} not homeomorphic to any effectively compact Polish \emph{space}. 
 Ng has recently announced a third proof of this fact  that uses a $\mathbf{0}'''$ priority construction. We note that, in contrast, every computably metrized Stone space is homeomorphic to an effectively compact one; this follows from the results in~\cite{HKSel-arxiv, HTMN-ta} as explained in the introduction of \cite{HTMN-ta}.  All these results mentioned above are very recent and rely on  advanced techniques.

The results discussed above separate the notions of effective metrizability, $\Delta^0_2$-metrizability, and effective compactness for Polish spaces viewed under homeomorphism.
There are also other notions of computability which frequently appear in the literature on computable topology, perhaps the most general of which is the point-free notion of a computable topological space. The notion merely requires existence of a computably enumerated base of the topology such that the intersection of two basic sets is uniformly computably enumerable.
But is this notion \emph{really} more general than, say, computable metrizability for compact Polish spaces?
Weihrauch and Grubba \cite{comptop} showed that, under a certain effective regularity assumption, a computable topological Polish space admits a computable metric.
The proof in \cite{comptop} builds a metric which is effectively compatible with the given computable topology via the identity operator.
Perhaps, we could drop the  effective regularity assumption and still  construct a computably metrized space \emph{homeomorphic} to the given computable topological space?
From the perspective of topology, this would make computable metrizability equivalent to computable topological presentability.
Note that we do not require the homeomorphism to be computable, and therefore it may seem that most geometrically natural potential counterexamples can be dynamically ``squished'' to form a computably metrized space.  Nonetheless, we use Stone spaces and Boolean algebras to prove:

\begin{theorem}\label{theo:no_comp_Polish}
	There exists a computable topological compact Polish space $S$ not homeomorphic to any computably metrized Polish space.
\end{theorem}

The reader will perhaps be surprised that the result is indeed new, for it may look like a theorem that should have been proven long ago.
The key step in the proof is a new effective version of Stone duality, Theorem~\ref{theo:Boolean},  which says that a Stone space $S$ is homeomorphic to a right-c.e.~metrized, effectively compact space if, and only if, the dual Boolean algebra admits a c.e.~presentation. It is well-known that right-c.e.~metric spaces are computable topological spaces; we will explain this in the preliminaries. Recall that  Feiner~\cite{feiner1970} constructed an example of a  c.e.~presentable Boolean algebra  which does not have a computable presentation.  Thus, Theorem~\ref{theo:Boolean} stated above combined with the result of Feiner  gives Theorem~\ref{theo:no_comp_Polish}. It also follows that our proof really gives more than is stated in Theorem~\ref{theo:no_comp_Polish}:

\begin{corollary} There exists a right-c.e.~metrized Polish space not homeomorphic to any computably metrized Polish space.
\end{corollary}

We leave open whether every computable topological compact Polish space is homeomorphic to a right-c.e.\ metrized one, but we of course conjecture that there should exist a counterexample.

\

Our second main result uses Stone spaces to test a new notion of computable categoricity for separable spaces.
Before we state the result, we remind the reader that
Pour El and Richards~\cite{PourElRich} were the first to investigate different computable separable spaces which are isometric but not computably isometric. The much more recent work \cite{Ilja} used isometric computable structures as a tool. 
Beginning with \cite{MelIso}, there has been a line of investigation focused on computable Polish spaces having a unique computable metrization  up to computable isometric isomorphism; 
such spaces are called \emph{computably categorical} (w.r.t.~isometries). For more results on computably categorical spaces,
see \cite{MelNg,McNellp,PhDBr,PhDMel}. 
It is often more natural to consider Polish spaces and especially Polish groups up to homeomorphism (resp., algebraic homeomorphism).
The recent paper \cite{pontr} introduces the notion of computable categoricity for profinite groups under computable algebraic homeomorphism.
The study of computably metrized spaces up to (not necessarily computable) homeomorphism was pioneered by  \cite{HKSel-arxiv,  HTMN-ta}.
 Nonetheless, computable categoricity for Polish spaces under \emph{computable} homeomorphism has not yet been studied in the literature. 
Our second main result classifies  Stone spaces which have a unique effectively compact presentation  up to \emph{computable} homeomorphism. More specifically, we prove:

\begin{theorem}\label{theo:comp-cat}
	Let $\mathcal{B}$ be a computable Boolean algebra. Then the following conditions are equivalent:
	\begin{itemize}
		\item[(a)] $\mathcal{B}$ is computably categorical.
		
		\item[(b)] The Stone space $\widehat{\mathcal{B}}$ is effectively categorical.
	\end{itemize}
\end{theorem}
\noindent Here a Polish space is effectively categorical if any pair of effectively compact presentations of the space is computably homeomorphic; we will elaborate these terms in the preliminaries. This result, even though it is not particularly difficult to prove, initiates the study of computable categoricity for spaces considered up to homeomorphism rather than up to isometry. The result also indicates that the notion of effective compactness is likely the ``right'' notion of computable presentability for Stone spaces when they are considered up to homeomorphism. 

Computable categoricity  of effectively compact Polish spaces and Polish groups up to computable homeomorphism is a wide open area. One naturally seeks to give purely topological characterisations of categoricity in natural classes of compact Polish spaces in the spirit of the following corollary:

\begin{corollary} A Stone space $\widehat{\mathcal{B}}$ is effectively categorical if, and only if, it has only finitely many isolated points.
	
\end{corollary}	 
\noindent The corollary of course follows immediately from Theorem~\ref{theo:comp-cat} and the well-known characterisation of computably categorical Boolean algebras~\cite{Goncharov-Dzgoev-80,Remmel81-BA}.
 Also, it would be interesting to see if there is a syntactical characterisation of relative categoricity (which needs to be defined formally) in the spirit of \cite{AshKn}. We leave these problems open.

\

Our third main result  applies Stone spaces to prove a theorem about computable Banach spaces. 
Let $X,Y$ be compact Polish spaces, and let $C(X; \mathbb{R})$ denote the Banach space of continuous functions $X \rightarrow \mathbb{R}$ under the supremum metric and pointwise operations.
Recall that the Banach–Stone theorem states that Banach spaces $C(X;\mathbb{R})$ and $C(Y;\mathbb{R})$ are isometrically isomorphic if, and only if,  $X$ and $Y$ are homeomorphic. 
This means that the (linear) isometry type of $C(X;\mathbb{R})$ is determined by the homeomorphism type of $X$, and vice versa. 
In personal communication with the third author, McNicholl has recently raised the question of whether this fact holds computably in the following sense. He observed that, for a computably metrized compact Polish space $X$,  $C(X;\mathbb{R})$ admits a computable Banach space presentation (all these notions will be formally defined later). Does the converse hold? More specifically:

\begin{center}
Does the computable presentability of $C(X;\mathbb{R})$ imply that $X$ is homeomorphic to a computable Polish space?
\end{center}

\noindent Although we suspect that the question above can likely be answered in negative by constructing a counterexample, we prove the following \emph{positive} result for totally disconnected spaces:

\begin{theorem}\label{thm1}
	Let $X$ be a separable Stone space and let $C(X;\mathbb{R})$ be the  Banach space of continuous functions $X \to \mathbb{R}$. Then the following are equivalent:
	
	\begin{enumerate}
\item  $C(X;\mathbb{R})$ has a presentation as a computable Banach space;

\item $X$ is computably metrizable.

\end{enumerate}	
\end{theorem}
We emphasise that in (1) we consider $C(X; \mathbb{R})$ up to isometric linear isomorphism, but in (2) we view $X$ up to homeomorphism.
It was recently proven in \cite{HTMN-ta} that (2) of Theorem~\ref{thm1}  is equivalent to computable presentability of  the Boolean algebra which is dual to $X$.
Thus, our theorem  combined with the main result of \cite{knight2000} and the aforementioned result from \cite{HTMN-ta},  gives the following peculiar consequence.
	
\begin{corollary} Suppose  $C(X;\mathbb{R})$ has a $low_4$ Banach space presentation. If  $X$ is a Stone space, then  $C(X;\mathbb{R})$ is isometrically isomorphic to a  computable Banach space.
	
\end{corollary}	

We leave open whether every low Banach space of the form $C(X;\mathbb{R})$, where $X$ is compact, has a computable presentation. This is, of course, closely related to the question of McNicholl that we stated above, and which we will also leave open for spaces which are not totally disconnected.

\section{Preliminaries}
\subsection{Effective metrizations of Polish spaces}
A Polish space $(M,d)$ is \emph{right-c.e.~presented}  or \emph{admits a right-c.e.~metric}  if there exists a sequence $(\alpha_i)_{i\in\omega}$ of $M$-points which is dense in $M$ and such that for every $i,j\in\omega$, the distance $d(\alpha_i, \alpha_j)$ is a right-c.e.~real, uniformly in $i$ and $j$. Formally, there is a c.e. set $W \subseteq \omega^{2}\times \mathbb{Q}$ such that for any $i$ and $j$,
\[
	\{ q\in \mathbb{Q} : d(\alpha_i, \alpha_j) < q\} = \{ q : (i,j,q)\in W\}.
\]
Note that the sequence $(\alpha_i)_{i \in \omega}$ may contain repetitions; equivalently, it is possible that $d(\alpha_i, \alpha_j) = 0$ for some $i,j$.

The definition of a left-c.e.~Polish~space is obtained from the notion of a right-c.e.~Polish space using the notion of a left-c.e.~real, mutatis mutandis. A Polish space is \emph{computably presented} or, perhaps more descriptively, \emph{computably metrizable} if there is a metric 
on the space which is both right-c.e.~and left-c.e.

Note that, strictly speaking, a computable or a right-c.e.~metrization of a space is a countable object $(\alpha_i)_{i\in\omega}$, but
we will usually identify a computable or a right-c.e.~metrization $(\alpha_i)_{i\in\omega}$ of space $M$  with its completion $\overline{(\alpha_i)_{i\in\omega}}$. We will also denote computable presentations by letters $X, Y, Z,\ldots$. The exact choice of notation for the dense set is not important.

\begin{remark}
Note that we intentionally did not emphasise whether we consider Polish spaces up to isometric isomorphism or under some other notion of similarity, such as, e.g., quasi-isometry or homeomorphism. Indeed, these will lead to non-equivalent notions. For example, for a real $\xi$, the space $[0, \xi]$ is isometrically isomorphic to a computably metrized space if, and only if $\xi$ is left-c.e. However, for any real $\xi$ this space is homeomorphic to the unit interval $[0,1]$ which is of course computably metrizable.

 Traditionally, Polish spaces in computable analysis have been viewed under isometric isomorphism; see, e.g., \cite{PourElRich}. In this paper we usually consider Polish spaces under homeomorphism, that is, a Polish space has a right-c.e.~presentation if it is \emph{homeomorphic} to the completion of a right-c.e.~metrized space. Nonetheless, we will emphasise this in most of the theorems and lemmas that we prove to make sure that there is no conflict of terminology. \end{remark}

 \subsection{Computable topological spaces}
 
\begin{definition}[see, e.g., Definition~4 of~\cite{comptop}]
	A \emph{computable topological space} is a tuple $(X,\tau,\beta,\nu)$ such that
	\begin{itemize}
		\item $(X,\tau)$ is a topological $T_0$-space,
		
		\item $\beta$ is a base of $\tau$,
		
		\item $\nu\colon \omega \to \beta$ is a surjective map, and
		
		\item there exists a c.e. set $W$ such that for any $i,j\in\omega$,
		\[
			\nu(i) \cap \nu(j) = \bigcup \{ \nu(k)\,\colon (i,j,k) \in W\}.
		\]
	\end{itemize}
\end{definition}

Let $(X,\tau,\beta,\nu)$ be a computable topological space. For $i\in\omega$, by $B_i$ we denote the open set $\nu(i)$. As usual, we identify basic open sets $B_i$ and their $\nu$-indices. 

\begin{definition}
A computable topological space $(X,\tau,\beta,\nu)$ is \emph{effectively compact} if it is equipped with an effective enumeration 
\[
	\{  \vec{D}^i = (D^i_0, D^i_1, \dots, D^i_{k_i}) \}_{i\in\omega}
\]
of all tuples of basic open sets such that $X = \bigcup_{j \leq k_i} D^i_{j}$. 
\end{definition}

The usual examples of computable topological spaces include computably metrized Polish spaces and right-c.e.~metrized Polish spaces. The latter is well-known; nonetheless, we decided to include a complete proof of this fact.

\begin{proposition}\label{prop:right-ce}
	Every right-c.e.\ Polish space is a computable topological space.
\end{proposition}
 
\begin{proof} 
Let $(M,d)$ be a right-c.e.~Polish space, and let $(\alpha_i)_{i\in\omega}$ be its sequence of special points. By $\tau$ we denote the metric topology of $(M,d)$. As usual, the base $\beta$ of $\tau$ contains basic open balls
\[
	B(\alpha_i,q) = \{ x\in M : d(\alpha_i,x) < q\},\ \ \  i\in\omega,\ q\in\mathbb{Q}^{+}.
\]
For $i\in\omega$ and $q\in\mathbb{Q}^{+}$, we put
$\nu(i,q) = B(\alpha_i,q)$. 

We prove that the tuple $(M,\tau,\beta,\nu)$ is a computable topological space. It is sufficient to establish the following: for any $i,j\in\omega$ and $q,r\in\mathbb{Q}^{+}$, one can (uniformly) effectively enumerate a set $X \subseteq \omega \times \mathbb{Q}^{+}$ such that 
\begin{equation}\label{equ:balls-01}
	 B(\alpha_i,q) \cap B(\alpha_j, r) = \bigcup \{B(\alpha_k, t) \,\colon (k,t) \in X \}.
\end{equation}

Our set $X$ is defined as follows: $X$ contains all pairs $(k,t)$ such that
\[
 	d(\alpha_i,\alpha_k) < q - t \text{ and } d(\alpha_j,\alpha_k) < r-t.
\]
Since the space $(M,d)$ is right-c.e., it is not hard to see that the set $X$ is c.e., uniformly in $i,j,q,r$.
If $(k,t)\in X$, then by using the triangle inequality, one can easily show that $B(\alpha_k,t)$ is a subset of $B(\alpha_i,q) \cap B(\alpha_j,r)$. 

Let $x$ be an arbitrary point from $U := B(\alpha_i,q) \cap B(\alpha_j, r)$. Choose positive rationals $\epsilon$ and $\delta$ such that $\epsilon < q - d(\alpha_i, x)$ and $\delta < r - d(\alpha_j, x)$. Since $U$ is open, one can find $k\in\omega$ and $t \in\mathbb{Q}^+$ such that $x\in B(\alpha_k, t) \subseteq U$ and $t < \min(\epsilon/2, \delta/2)$. Then we have
\[
	d(\alpha_i,\alpha_k) \leq d(\alpha_i,x) + d(\alpha_k,x) < (q - \epsilon)  + t < q - \epsilon/2 < q - t.
\]
Therefore, $(k,t)$ belongs to $X$, and the set $X$ satisfies~(\ref{equ:balls-01}). Hence, $(M,\tau,\beta,\nu)$ is a computable topological space. 
\end{proof}

\

\subsection{Computable Banach spaces} Let $X$ be a computable topological space.
For a point $x\in X$, its \emph{name} is the set
\[
	N^x = \{ i\in\omega\,\colon x\in B_i\}.
\]
An \emph{open name} of an open set $U \subseteq X$ is a set $W\subseteq \omega$ such that 
\[
	U = \bigcup_{i\in W} B_i.
\]

\begin{definition}
	Let $X$ and $Y$ be computable topological spaces. A function $f\colon X \to Y$ is \emph{effectively continuous} if there is a c.e.\ family $F\subseteq \mathcal{P}(X) \times \mathcal{P}(Y)$ of pairs of (indices of) basic open sets such that:
	\begin{itemize}
		\item[(C1)] for every $(U,V) \in F$, we have $f(U) \subseteq V$;
		
		\item[(C2)] for every point $x \in X$ and every basic open $E\ni f(x)$ in $Y$, there exists a basic open $D\ni x$ in $X$ with $(D,E) \in F$.
	\end{itemize}
\end{definition}

The elementary fact below is well-known; see, e.g.,  Lemma~2.7 of~\cite{MM-groups-18}.

\begin{lemma} \label{lem:effective-continuous}
	Let $f\colon X\to Y$ be a function between computable topological spaces. Then the following conditions are equivalent:
	\begin{enumerate}
		\item $f$ is effectively continuous.
		
		\item There is an enumeration operator $\Phi$ that on input an open name of an open set $V$ in $Y$ lists an open name of the set $f^{-1}(V)$ in $X$.
		
		\item There is an enumeration operator $\Psi$ that given the name of a point $x\in X$, enumerates the name of $f(x) \in Y$.
	\end{enumerate}
\end{lemma}

\ 

 The definition below is equivalent to the standard definition from \cite{PourElRich}.

\begin{definition}
A separable (real)  Banach space $\mathcal{B}$ is \emph{computably presented} if it is isometrically (linearly) isomorphic to a computably metrized Polish space in which the operations  $+$ and scalar multiplication $(r \cdot)_{r \in \mathbb{Q}}$ become uniformly computable with respect to the metric (more precisely, with respect to the computable topology induced by the metric).
\end{definition}

\

It is well-known that any self-isometry of a Banach space has to be affine, i.e., it can shift the origin $0$ but does respect the operation up to a translation, thus we do not really have to emphasise that the isometry in the definition has to be linear as long as it maps zero to zero. We however do emphasise that in this paper we view Banach spaces up to isometry and not up to homeomorphism.

\

\subsection{Effectively compact presentations}
For an open subset $U$ of a computable Polish space $(M,d)$, a \emph{finitary name} of $U$ is a sequence $\vec{C} = (C_0, C_1,\dots, C_k)$ of basic open balls such that $U = \bigcup_{i \leq k} C_i$. Note that any finitary name of $U$ (if it exists) is its open name.

\begin{definition}
A \emph{computable compact presentation} or an \emph{effectively compact presentation} of a Polish space $M$ is a computable metrization of $M$,
which is effectively compact.
\end{definition}

 We introduce the following notion of effective categoricity for effectively compact Polish spaces.

\begin{definition}
	We say that an effectively compact Polish space $M$ is \emph{effectively categorical} (or computably categorical with respect to effectively compact presentations) if  for any pair of effectively compact  $X$ and $Y$ \emph{homeomorphic} to $M$, there is an effectively continuous surjective homeomorphism from $X$ onto $Y$.
\end{definition}

Before proceeding to the new results, we observe the following useful fact:

\begin{remark}\label{remark:Brattka}
	Let $X$ and $Y$ be effectively compact presentations of $M$. By a result of Brattka (see Corollary~6.5 of~\cite{Brattka-08}), if $f$ is an effectively continuous surjective homeomorphism from $X$ onto $Y$, then its inverse $f^{-1}$ is an effectively continuous surjective homeomorphism from $Y$ onto $X$. See also Section~6.2 of~\cite{Ilj-Kihara} for a discussion.
\end{remark}

\begin{remark}
We will not develop the theory of computably compact (effectively compact) spaces any further.  Our treatment of computable Stone duality is self-contained, however, assuming various results from the two large recent surveys
\cite{Ilj-Kihara} and \cite{EffedSurvey} it can potentially be made more compact (no pun intended). This is explained in detail in  \cite{EffedSurvey}.  As suggested by the anonymous referee, our results on effective Stone duality can likely be extended to a more general setting using the theory of represented spaces; see \cite{Arara} for a comprehensive introduction. 
Furthermore, as was further noted by the referee, one could take a slightly more abstract approach to Stone duality following some category-theoretic ideas that can be extracted from \cite{TaylorStuff}. 
\end{remark}


\section{A computable topological space not homeomorphic to a computably metrized one}

Recall that a c.e.~presentation of a countably infinite  Boolean algebra is its isomorphic copy of the form $\mathcal{F}/ I$, where $\mathcal{F}$ is the countable atomless Boolean algebra and $I$ is its c.e.~ideal \footnote{Equivalently, it can be viewed as a pre-structure $(\omega,  \cup, \cap, \overline{\,\cdot\,},0,1, =)$ upon the domain of $\omega$ such that the operations $ \cup, \cap, \overline{\,\cdot\,}$ are computable, but the equality (the congruence) is merely computably enumerable. Note that the quotient of this structure by the c.e.~congruence can be finite. The notion of a c.e.-presented Boolean algebra will be extended in the proof below.}.

The plan of the proof of Theorem~\ref{theo:no_comp_Polish} is as follows. We will prove the new effective version of Stone duality stated below. 

\begin{theorem}\label{theo:Boolean}
	Let $B$ be an at most countable Boolean algebra, and let $\widehat{B}$ be the space of its ultrafilters. Then the following conditions are equivalent:
	\begin{itemize}
		\item[(a)] $B$ has a c.e.\ presentation,
		
		\item[(b)] $\widehat{B}$ admits a compatible, complete right-c.e.\ metric such that the induced computable topological space is effectively compact. 
	\end{itemize}
\end{theorem}

Feiner~\cite{feiner1970} constructed a c.e.~presentable Boolean algebra $B$ such that $B$ does not have computable copies. By Theorem~\ref{theo:Boolean}, one can assume that the Polish space $\widehat{B}$ is right-c.e. By Proposition~\ref{prop:right-ce}, $\widehat{B}$ is a computable topological space. 

On the other hand, since $B$ is not computably presentable, Theorem~1.1 of~\cite{HTMN-ta} (to be discussed) implies that $\widehat{B}$ is \emph{not} homeomorphic to a computable Polish space. Therefore, the space $\widehat{B}$ satisfies Theorem~\ref{theo:no_comp_Polish}. 

In the remainder of the section, we prove Theorem~\ref{theo:Boolean}.



\subsection{Proof of Theorem~\ref{theo:Boolean}} The case when $B$ is finite  is trivial. Thus, throughout the rest of the proof  we assume that the Boolean algebra is countably infinite.
First, we briefly discuss the techniques which we will use in the proof.

Let $T$ be a subtree of $2^{<\omega}$. As usual, $[T]$  denotes the set of all infinite paths through $T$. We say that $T$ is a \emph{pruned tree} if for any $\sigma \in T$, there is a path $x\in [T]$, which goes through $\sigma$.

Boolean algebras are treated as structures in the language $L_{BA} = \{ \cup, \cap, \overline{\,\cdot\,},0,1\}$.
Consider an extended language $L' = L_{BA} \cup \{ E^2\}$. Let $n$ be a non-zero natural number. We say that an $L'$-structure $\mathcal{C}$ is a \emph{$\Sigma^0_n$-presentation} of a Boolean algebra $\mathcal{B}$ if $\mathcal{C}$ satisfies the following conditions:
\begin{enumerate}
	\item $\mathrm{dom}(\mathcal{C}) = \omega$,
	
	\item the $L_{BA}$-reduct of $\mathcal{C}$ is a computable structure,  

	\item$E\in \Sigma^0_n$, and $E$ is a congruence of the $L_{BA}$-reduct of $\mathcal{C}$,
	
	\item the quotient $L_{BA}$-structure $\mathcal{C}/ E$ is isomorphic to $\mathcal{B}$.
\end{enumerate}
A $\Pi^0_n$-presentation of a Boolean algebra is defined in a similar way. We will use the following results of Odintsov and Selivanov~\cite{OdSel-89} (see also Section~2.4 of~\cite{HKSel-arxiv} for more details):

\begin{proposition}[{\cite{OdSel-89}}] \label{prop:OdSel}
	Let $\mathcal{B}$ be a countable Boolean algebra.
	\begin{enumerate}
		\item $\mathcal{B}$ has a computable copy if and only if $\mathcal{B}$ is isomorphic to the Boolean algebra of clopen subsets of $[T]$ for some computable pruned tree $T$ (Lemma~3 of~\cite{OdSel-89}).
	
		\item $\mathcal{B}$ has a c.e.\ presentation iff $\mathcal{B}$ is isomorphic to the algebra of clopen subsets of $[T]$ for a co-c.e.\ pruned tree $T$ (Lemma~3 of~\cite{OdSel-89}).
		
		\item If $\mathcal{B}$ has a $\Pi^0_2$-presentation, then $\mathcal{B}$ admits a c.e.\ presentation (Corollary~2 of~\cite{OdSel-89}).
	\end{enumerate}
\end{proposition}

Recall that for a Boolean algebra $\mathcal{B}$, by $\widehat{\mathcal{B}}$ we denote its Stone space, i.e., the space of its ultrafilters. Harrison-Trainor, Melnikov, and Ng~\cite{HTMN-ta} established the following effective version of Stone duality:

\begin{proposition}[Theorem~1.1 of {\cite{HTMN-ta}}] \label{prop:HTMN}
	For a countable Boolean algebra $\mathcal{B}$, the following conditions are equivalent:
	\begin{enumerate}
		\item $\mathcal{B}$ has a computable copy,
		
		\item the space $\widehat{\mathcal{B}}$ is homeomorphic to a computable Polish space.
	\end{enumerate}
\end{proposition}

We proceed to the proof of Theorem~\ref{theo:Boolean}.


\subsubsection{Proof of (a)$\Rightarrow$(b)} 

Suppose that a Boolean algebra $\mathcal{B}$ has a c.e.~presentation. By Proposition~\ref{prop:OdSel}, one can choose a co-c.e.\ pruned tree $T$ such that $\mathcal{B}$ is isomorphic to the algebra of clopen subsets of $[T]$.

We define a right-c.e.\ Polish presentation $(M,d)$ for the space $\widehat{\mathcal{B}}$. We put $M=[T]$, and the distance $d$ is induced by the standard ultrametric on the Cantor space $2^{\omega}$. We build a dense sequence $(\alpha_i)_{i\in\omega}$ inside $(M,d)$~--- our construction needs to ensure that the distances $d(\alpha_i,\alpha_j)$ are uniformly right-c.e. Note that in general, a special point $\alpha_i$ could be equal to $\alpha_j$ for $i\neq j$.

 Since the tree $T$ is co-c.e., we can fix an effective enumeration of its complement: 
 \[
 	2^{<\omega} \setminus T = \bigcup_{s\in\omega} V_s, \text{ where } V_s \subseteq V_{s+1}.
\] 
Since $T$ is a tree and thus has to be closed under prefixes, we can further assume that each $V_s$ is a finite union of sets of the form $\{\tau: \tau \supseteq \sigma  \}$ for some finite collection of such strings $\sigma$.
Set $\Gamma_s$ be equal to the set of all finite strings of length at most $s$ in $2^{<\omega} \setminus V_s$. Refine this sequence to a sequence $(T_s)_{s \in \omega}$
so that $T_s$ is a finite tree, and $T_{s+1}$  and $T_s$ differ by at most one string.
It should be clear that the sequence $(T_s)_{s \in \omega}$ is uniformly computable and additionally, the following properties are satisfied:

\begin{itemize}	

\item[$(i)$] if $\sigma\in T_s$, then $|\sigma| \leq s$;

	\item[$(ii)$] $\sigma \in T$ if and only if $(\exists s_0) (\forall s \geq s_0) (\sigma \in T_s)$.
	
\end{itemize}

We are ready to construct our dense sequence $(\alpha_i)_{i\in\omega}$. 
The idea is as follows. The strings of the form $\sigma 0^{\omega}$ in $T_s$ will be used to list a dense set. If $\sigma \in T_s$ leaves $T_{s+1}$, then we declare  $\sigma 0^{\omega}$ equal (in terms of the distance) to some carefully chosen currently closest $\tau 0^{\omega}$, where $\tau  \in T_{s+1}$.  This process will eventually stabilize at every level of the tree, and thus 
the resulting metric will be well-defined and right-c.e.

We identify a finite string $\sigma$ with $\sigma 0^{\omega} \in 2^{\omega}$ to make sense of $d(\sigma, \tau)$ when $\sigma, \tau \in 2^{<\omega}$.
We can also view $\sigma$ as a function $\omega \rightarrow \{0,1\}$ with finite support. Recall also that $[T]$ is non-empty (in fact,  infinite).

\

At stage $0$, let $\alpha_i = \sigma_i $, where $(\sigma_i : i \in \omega)$ is some fixed uniformly effective enumeration of all finite strings in $2^{<\omega}$.
Also, define $d(\alpha_i, \alpha_j)$ equal to the least common prefix distance ultrametic inherited from $2^{\omega}$.

\

At stage $s+1$, assume $\sigma \in T_{s} \setminus T_{s+1} = \{\sigma\} \neq \emptyset$. That is, assume $\sigma$ `leaves' $T$ at stage $s+1$. (If no string leaves $T$ then do nothing.) 
Let $\tau \in T_{s+1}$ be so that it has the longest possible common prefix with $\sigma$, and among such $\sigma$ it is the smallest under the Kleene–Brouwer order on the strings.
For each $j$ such that $\alpha_j = \sigma $ and every $k$ such that $\alpha_k = \tau$, declare $d(\alpha_j, \alpha_k) = 0$. Equivalently, declare $\alpha_j = \alpha_k$.
For any $i$, set $d(\alpha_j, \alpha_i) = d(\alpha_k, \alpha_i)$, and proceed to the next stage.

\

This concludes the construction. 

\

Let $\alpha_i[s]$ be the string $\sigma 0^\omega$ so that, at stage $s$, $\alpha_i = \sigma$.
It should be clear that every bit of $\alpha_i[s]$ can change only finitely many times; this is because strings that leave $T$ will never be introduced back in $T$. Hence, $\alpha_i = \lim_s \alpha_i[s]$ is well-defined.
Equivalently, the sequence converges in $2^{\omega}$ to a point, and this point has to be in $[T]$. Furthermore,  the sequence $\{\alpha_i\}_{i\in\omega}$ is dense in $([T],d)$.

The properties of the construction ensure that
\[
	d(\alpha_i[s+1],\alpha_j[s+1]) \leq d(\alpha_i[s], \alpha_j[s]) \text{ for all } i,j,s.
\] 
Therefore, for a rational $q$, the condition $d(\alpha_i,\alpha_j) < q$ holds if and only if there is a stage $s$ such that $d(\alpha_i[s], \alpha_j[s]) < q$. We deduce that the reals $d(\alpha_i,\alpha_j)$ are uniformly right-c.e., and the Stone space $\widehat{\mathcal{B}}$ has a right-c.e.\ Polish presentation.

Now it remains to show that $\widehat{\mathcal{B}}$ is effectively compact (as a topological space). 
The desired effective enumeration of finite open covers is constructed as follows. 
At a stage $s$, we add a tuple of basic open balls if it seems to cover the whole space $[T]$, according to our current best guess. More formally, 
for a potential cover 
\[
	\vec{B} = ( B(\alpha_{i_0}, r_0), \dots,  B(\alpha_{i_k}, r_k)),
\]
we can check that $\vec{B}$ covers every node in the finite tree $T_s$. It follows by induction on the stage of the construction that $\vec{B}$ will remain a cover of $[T]$ at every later stage: we never introduce new points outside of the cover, and all the points which are already in the cover will remain inside it (since distances between points can only get smaller).

We argue that all finite covers (by basic clopen balls) will be eventually listed in this enumeration. Fix one such cover
\[
	\vec{B} = (B(\alpha_{i_0}, 2^{-l_0}), \dots,  B(\alpha_{i_k}, 2^{-l_k})).
\]
Consider a stage $s_0$ such that $\alpha_i[s] \upharpoonright m = \alpha_i \upharpoonright m$
 for $m = \max(l_0,\dots,l_k)$ and all $s\geq s_0$. Then $\vec{B}$ is a cover at every stage $s\geq s_0$. Therefore, it will be listed.

\begin{remark}
We suspect that the argument above can perhaps be extended to an arbitrary $\Pi^0_1$ closed subset of a computably compact  $K$ using the computable version of Hausdorff--Alexandroff Theorem (see~\cite[Prop.~4.1]{BBP}, and see \cite[Thm.1(vii)]{EffedSurvey} for two alternative proofs). 
\end{remark}


\subsubsection{Proof of (b)$\Rightarrow$(a)}

Suppose that the Stone space $\widehat{\mathcal{B}}$ has a right-c.e.~Polish presentation $(M,d)$. Let $(\alpha_i)_{i\in\omega}$ be its sequence of special points. By Proposition~\ref{prop:OdSel}, it is sufficient to build a $\Pi^0_2$-pre\-sen\-ta\-ti\-on $\mathcal{C}$ of the Boolean algebra $\mathcal{B}$.

We employ the tree-basis technique thoroughly explained in the monograph~\cite{Gon-97}. We outline it here. The full binary tree $T = 2^{<\omega}$ can be treated as a computable tree-basis of a computable atomless Boolean algebra $\mathcal{A}$. We declare that the $L_{BA}$-reduct of our presentation $\mathcal{C}$ is equal to $\mathcal{A}$.

We fix an effective enumeration \[
	\big\{  \vec{B}^i = (B^i_0, B^i_1, \dots, B^i_{k_i}) \big\}_{i\in\omega}
\] 
of all possible finitary names in the space $M$. 
For a basic open ball $B$, by $r(B)$ we denote the radius of $B$, and $c(B)$ denotes the center of $B$. 

	Let $U$ be an open set with a finitary name $\vec{B} = (B_0,B_1,\dots,B_k)$. Then the following two conditions are equivalent:
	\begin{itemize}
		\item[(i)] $U$ is clopen and $U \neq M$.
		
		\item[(ii)]	There exists another open set $V$ with a finitary name $\vec{D} = (D_0,D_1,\dots,D_{\ell})$ such that:
		\begin{itemize}
			\item[(a)] $U \cup V = M$, and
			
			\item[(b)] the balls $B_i$ and $D_j$ do not intersect, for all $i$ and $j$.
		\end{itemize}
	\end{itemize}
	All (clopen) sets $U$ satisfying Condition~(ii) can be listed using $\mathbf{0}'$: Condition~(a) is $\Sigma^0_1$ by effective compactness, and~(b) is $\Pi^0_1$, since it is equivalent to 
	\[
		\neg \exists k [ d(\alpha_k, c(B_i)) < r(B_i) \ \&\  d(\alpha_k, c(D_j)) < r(D_j)].
	\]

We fix a $\mathbf{0}'$-effective list $(U_i)_{i\in\omega}$ of all clopen $U_i$ satisfying Condition~(ii) 
(i.e., we fix a $\mathbf{0}'$-computable function, which maps $i\in\omega$ to a strong index of a finitary name of a clopen set that we denote by $U_i$).

Every node $\sigma\neq\emptyset$ of the tree $T$ is associated with a clopen set $V_{\sigma}$, which is defined as follows:
\[
	V_{\sigma} = U^{\sigma(0)}_0 \cap U^{\sigma(1)}_1 \cap \dots \cap U^{\sigma(|\sigma|-1)}_{|\sigma|-1},
\]
where $U^1 = U$ and $U^0 = \overline{U} = M\setminus U$. We define $V_{\emptyset} = M$.

Now the structure $\mathcal{A}$ can be identified with the \emph{formal algebra} of all clopen subsets of $M$: The family $\{ V_{\sigma} : \sigma \in T\}$ can be treated as a tree-basis for the algebra $\tilde{\mathcal{B}}$ of clopen subsets of $M$. The formal $\mathcal{A}$-operations $\cup$, $\cap$, and $\overline{\,\cdot\,}$, induced by this tree-basis, are precisely the standard set-theoretic operations inside $\tilde{\mathcal{B}}$. Note that in this formal algebra, a clopen set  can have many names: e.g., it can be the case that $U_0 = U_1$~--- this implies that $V_1 = V_{11} = V_{11} \cup V_{10}$.

A congruence relation $E$ on the formal algebra $\mathcal{A}$ can be defined as follows. Given strings $\sigma$ and $\tau$, one can computably find a tuple $\xi_0,\xi_1,\dots,\xi_m \in T$ such that 
\[
	V_{\sigma}\triangle V_{\tau} \overset{\mathrm{df}}{=} (V_{\sigma}\cap \overline{V_{\tau}}) \cup (\overline{V_{\sigma}} \cap V_{\tau}) = \bigcup_{i\leq m} V_{\xi_i}.
\]
Then define:
\[
	V_{\sigma} \not\sim_E V_{\tau}  \ \Leftrightarrow\ \emptyset \not\sim_E \bigcup_{i\leq m} V_{\xi_i} \ \Leftrightarrow\ (\exists i\leq m) (V_{\xi_i} \neq \emptyset) \ \Leftrightarrow\ \bigcup_{i\leq m} V_{\xi_i} \neq \emptyset.
\]

	The quotient structure $\mathcal{A}/E$ is isomorphic to the algebra of all clopen subsets of $M$.

\begin{claim}\label{claim:aux-not-empty}
	The condition $V_{\sigma} \neq \emptyset$ is $\Sigma^0_2$. 
\end{claim}
\begin{proof}
Let $B$ be a basic open ball, and let $i\in\omega$. Since our space is right-c.e., checking whether $\alpha_i \in B$ is a $\mathbf{0}'$-computable procedure, which is uniform in $B$ and $i$. This fact implies the following: given $i$ and a clopen set $U$, which is described as a Boolean combination of basic open balls, one can $\mathbf{0}'$-effectively check whether $\alpha_i$ belongs to $U$. Note that $V_{\sigma} \neq \emptyset$ if and only if there is $i\in\omega$ such that $\alpha_i \in V_{\sigma}$ (since $V_{\sigma}$ is open).

Recall that
\begin{equation}\label{equ:aux-not-empty}
	V_{\sigma} = U^{\sigma(0)}_0 \cap U^{\sigma(1)}_1 \cap \dots \cap U^{\sigma(|\sigma|-1)}_{|\sigma|-1},
\end{equation}
where each $U_i^{\sigma(i)}$ is either a finite union of basic open balls, or the complement of such a union.
Using $\mathbf{0}'$, given $\sigma$, we compute all finitary names of $U_i$ from the decomposition~(\ref{equ:aux-not-empty}).
If there exists an $\alpha_i \in V_{\sigma}$, then $\mathbf{0}'$ will eventually witness it. We deduce that checking the condition $V_{\sigma} \neq \emptyset$ is $\mathbf{0}'$-c.e. 
\end{proof}

Claim~\ref{claim:aux-not-empty} implies that the congruence $E$ is $\Pi^0_2$. 
Therefore, $\mathcal{C} = (\mathcal{A},E)$ is a $\Pi^0_2$-presentation of the original algebra $\mathcal{B}$, and thus $\mathcal{B}$ admits a c.e. presentation by (3) of  Proposition~\ref{prop:OdSel}.
\qed




\section{Categoricity for Stone spaces. Proof of Theorem~\ref{theo:comp-cat}.}

Recall that Theorem~\ref{theo:comp-cat} says that, for a computable Boolean algebra $\mc{B}$, $\mc{B}$ is computably categorical if and only if the Stone space $\widehat{\mathcal{B}}$ is effectively categorical.

\begin{proof}
\underline{(b)$\Rightarrow$(a).} Suppose that the space $\widehat{\mathcal{B}}$ is effectively categorical, meaning that each pair of effectively compact presentations of the space are computably homeomorphic.

Let $\mathcal{A}$ be a computable copy of the algebra $\mathcal{B}$. Following the proof of Proposition~\ref{prop:OdSel}.(1), one can build a computable pruned tree $T_A$ such that $\mathcal{A}$ is isomorphic to the Boolean algebra $\mathrm{Clop}([T_A])$ of all clopen subsets of $[T_A]$. The metric $d$ on $[T_A]$ is induced by the standard ultrametric on $2^{\omega}$, and it is not hard to recover an effectively compact presentation of $([T_A],d)$~--- see, e.g., Theorem~2.9 of~\cite{HKSel-arxiv} for more details. Let $M_{\mathcal{A}}$ denote this effectively compact presentation.

The transformation $\mathcal{A} \mapsto M_{\mathcal{A}}$ has the following nice properties. Given an element $a\in\mathcal{A}$ such that $a\neq 0_{\mathcal{A}}$, one can effectively recover a finite tuple $\sigma_0,\sigma_1,\dots,\sigma_k\in T_A$ such that the natural isomorphism from $\mathcal{A}$ onto $\mathrm{Clop}([T_A])$ acts as follows:
\[
	a \mapsto  V_a := \{ x\in[T_A]\,\colon x \text{ goes through one of }  \sigma_i \}.
\]
Moreover, one can effectively find a finitary name $\vec{B}^a$ for the clopen set $V_a$.

Given two tuples $\sigma_0,\sigma_1,\dots,\sigma_k$ and $\tau_0,\tau_1,\dots,\tau_{\ell}$, one can effectively check whether the sets
\[
	Z_{\bar \sigma} = \{ x\in [T_A]\,\colon x \text{ goes through one of }  \sigma_i \} \text{ and }	
	Z_{\bar \tau} = \{ y\in [T_A]\,\colon y \text{ goes through one of }  \tau_j \}
\]
are equal or not. The idea behind this effective procedure can be illustrated by the following example\footnote{One can design a more elegant, although not self-contained, procedure using the theory of computably compact sets. We give a more brute-force and explicit way to decide intersection.}.

Consider $\sigma$ and $\tau_0,\tau_1$ from $T_A$. Then there are three possible cases:
\begin{enumerate}
	\item If one of $\tau_i$-s is incomparable with $\sigma$, then there is infinite path $x$, which goes through this $\tau_i$, but does not go through $\sigma$.
	
	\item Suppose that both $\tau_i$ are comparable with $\sigma$, $\tau_0 \subseteq \sigma$, and $|\tau_0| \leq |\tau_1|$. Then (since the tree $T_A$ is a computable subtree of $2^{<\omega}$) we can effectively find all strings $\xi_k$ such that $\xi_k\supseteq \tau_0$ and $|\xi_k| = |\sigma|$. If among them, there is a string $\xi_k \neq \sigma$, then there is a path $x$ going through $\xi_k \supseteq \tau_0$, but not through $\sigma$. Otherwise, we have $Z_{\sigma} = Z_{\tau_0,\tau_1}$.
	
	\item The last remaining case is when we have $\tau_0 \supset \sigma$ and $\tau_1 \supset \sigma$. We find all strings $\zeta_k$ such that $\zeta_k \supset \sigma$ and $|\zeta_k| = \max\{ |\tau_0|, |\tau_1|\}$. If among them, there is a string $\zeta_k$ such that $\zeta_k \nsupseteq \tau_0$ and $\zeta_k \nsupseteq \tau_1$, then there is a path going through $\zeta_k \supset \sigma$, but not hitting $\tau_0$ and $\tau_1$. Otherwise, the sets $Z_{\sigma}$ and $Z_{\tau_0,\tau_1}$ are equal.
\end{enumerate}

\ 

Now we are ready to prove that the algebra $\mathcal{B}$ is computably categorical. Let $\mathcal{A}$ and $\mathcal{C}$ be computable copies of $\mathcal{B}$. Consider the compact presentations $M_{\mathcal{A}}$ and $M_{\mathcal{C}}$, and fix an effectively continuous surjective homeomorphism $f$ acting from $M_{\mathcal{C}}$ onto $M_{\mathcal{A}}$. We construct a computable isomorphism $g$ from $\mathcal{A}$ onto $\mathcal{C}$.

Let $b$ be an element from $\mathcal{A}$ such that $b\not\in \{ 0_{\mathcal{A}}, 1_{\mathcal{A}}\}$. We effectively recover the finitary names $\vec{B}^b$ and $\vec{B}^{\overline{b}}$ for the clopen sets $V_b$ and $V_{\overline{b}}$ representing the element $b$ and its complement $\overline{b}$. 

By Lemma~\ref{lem:effective-continuous}, we fix an enumeration operator $\Phi$, which given an open name of $V\subseteq M_{\mathcal{A}}$, outputs an open name of the set $f^{-1}(V) \subseteq M_{\mathcal{C}}$.

We effectively enumerate the open names:
\begin{gather*}
	\Phi(\vec{B}^{b}) = \{ C_0, C_1, C_2, \dots\},\quad 
	\Phi(\vec{B}^{\overline{b}}) = \{ D_0, D_1, D_2,\dots\}.
\end{gather*}
Note that both of these lists could be \emph{infinite}. On the other hand, the sets $f^{-1}(V_b) = \Phi(\vec{B}^{b})$ and $f^{-1}(V_{\overline{b}}) = \Phi(\vec{B}^{\overline{b}})$ form a splitting of the space $M_{\mathcal{C}}$. Hence, since the presentation $M_{\mathcal{C}}$ is compact, eventually we will find a number $k$ such that $C_0,C_1,\dots C_k,D_0,D_1,\dots,D_k$ form an open cover of $M_{\mathcal{C}}$. This means that 
\[
	f^{-1}(V_b) = \bigcup_{i\leq k} C_i,\quad f^{-1}(V_{\overline{b}}) = \bigcup_{i\leq k} D_i.
\]

Given the basic open sets $C_0,C_1,\dots,C_k$, we effectively recover a tuple $\tau_0,\tau_1,\dots,\tau_m \in T_C$ such that
\[
	f^{-1}(V_b) = \{ x\in [T_C]\,\colon x \text{ goes through one of } \tau_i\} = Z_{\bar \tau}.
\]
Recall that the procedure of checking whether $Z_{\bar \tau}$ equals $Z_{\bar \sigma}$ is effective (see above). This implies that we can effectively find an element $d\in \mathcal{C}$ with $f^{-1}(V_b) = V_d$. We put $g(b) := d$.

Clearly, the constructed map $g$ is computable and well-defined. It is not hard to show that $g$ is an isomorphism from $\mathcal{A}$ onto $\mathcal{C}$.

\ 

\underline{(a)$\Rightarrow$(b).} Suppose that the algebra $\mathcal{B}$ is computably categorical. It is known~\cite{Goncharov-Dzgoev-80,Remmel81-BA} that $\mathcal{B}$ has only finitely many atoms. Without loss of generality, we assume that $\mathcal{B}$ is infinite. 

First, we give a detailed proof for the case when $\mathcal{B}$ is a countable atomless Boolean algebra. After that, we discuss the modifications needed for the general case.

Let $B(c,r)$ denote the basic open ball with center $c$ and radius $r$. If $D$ is a basic open ball, then by $r(D)$ we denote its radius, and by $c(D)$ we denote its center.

Consider two open sets
\[
	U = \bigcup_{i\leq k} B_i \ \text{and}\ V = \bigcup_{j\leq \ell} C_{j},
\]
where $B_i$ and $C_j$ are basic open balls. We say that $U$ and $V$ are \emph{formally non-intersecting} if for all $i$ and $j$, we have
\[
	d(c(B_i), c(C_j)) > r(B_i) + r(C_j).
\]
It is clear that formally non-intersecting $U$ and $V$ satisfy $U\cap V = \emptyset$.

We say that $U$ is \emph{formally included} into $V$ if for each $i\leq k$, there is an index $j_i \leq \ell$ such that
\[
	d( c(C_{j_i}), c(B_i) ) + r(B_i) < r(C_{j_i}).
\]
If $U$ is formally included into $V$, then $U\subseteq V$.

In the lemma below, we identify clopen subsets of $\mathcal{M}$ with their finitary names.

\begin{lemma}[This is similar to Lemma~2.4 of~\cite{HKSel-arxiv}] \label{lem:aux-tree-splitting}
	Let $\mathcal{M}$ be a compact presentation of the space $\widehat{\mathcal{B}}$. Then one can effectively build a computable tree $T_{\mathcal{M}} \subset \omega^{<\omega}$ and a sequence $\{ U^{\mathcal{M}}_{\sigma} \,\colon \sigma \in T_{\mathcal{M}}\}$ of non-empty clopen sets, each containing infinitely many elements, such that:
	\begin{itemize}
		\item[(a)] The tree $T_{\mathcal{M}}$ is finitely branching, and its branching function 
		\[
			b_{\mathcal{M}} (\sigma) = \mathrm{card}(\{ n\in\omega\,\colon \sigma\widehat{\ } n \in T_{\mathcal{M}}\})
		\]
		is computable.
		
		\item[(b)] For any $\sigma \in T_{\mathcal{M}}$, $b_{\mathcal{M}}(\sigma) \geq 2$.
		
		\item[(c)] Let $\sigma$ and $\tau$ be elements of $T_{\mathcal{M}}$.
		\begin{itemize}
			\item[(c.1)] If $\sigma\neq\emptyset$, then each basic open ball $C$ taken from the finitary name of the set $U^{\mathcal{M}}_{\sigma}$ satisfies $r(C) \leq \frac{1}{2^{|\sigma|}}$. If $\sigma = \emptyset$, then $U^{\mathcal{M}}_{\emptyset} = \widehat{\mathcal{B}}$.
		
			\item[(c.2)] If $\sigma$ and $\tau$ are siblings, then the sets $U^{\mathcal{M}}_{\sigma}$ and $U^{\mathcal{M}}_{\tau}$ are formally non-intersecting.
			
			\item[(c.3)] If $\tau$ is a child of $\sigma$, then the set $U^{\mathcal{M}}_{\tau}$ is formally included into $U^{\mathcal{M}}_{\sigma}$.
			
			\item[(c.4)] We have
			\[
				U^{\mathcal{M}}_{\sigma} = \bigcup \{ U^{\mathcal{M}}_{\xi}\,\colon \xi \in T_{\mathcal{M}},\ \xi \text{ is a child of } \sigma\}.
			\]
		\end{itemize}
	\end{itemize}
\end{lemma}
\begin{proof}
	For the sake of completeness, here we give a sketch of the proof. The tree $T_{\mathcal{M}}$ is built by induction on $s\in\omega$. At a stage $s$, we define all vertices $\sigma \in T_{\mathcal{M}}$ with $|\sigma| = s$. 
	
	At the stage $0$, by the compactness of $\widehat{\mathcal{B}}$, we non-uniformly choose a rational $R$ such that $B(\alpha_0,R) = \widehat{\mathcal{B}}$. We put $U^{\mathcal{M}}_{\emptyset} := B(\alpha_0, R)$.
	
	\emph{Stage $s+1$.} Recall that the compact presentation $\mathcal{M}$ provides an effective enumeration of all finitary names of the space $\widehat{\mathcal{B}}$. By going through this enumeration, we search for a finite splitting of each $U_{\xi}^{\mathcal{M}}$, where $\xi \in T_{\mathcal{M}}$ and $|\xi| = s$, which satisfies the conditions (c.1)--(c.3). More formally, we search for a finitary name (provided by the enumeration), which encodes the union
	\[
		\widehat{\mathcal{B}} = \bigcup_{\xi \in T_{\mathcal{M}},\, |\xi| = s} \left( \bigcup_{i\leq m_{\xi}} \left( \bigcup_{j\leq n_{\xi,i}} E_{\xi,i,j}\right)\right),
	\]
	where $m_{\xi} \geq 1$, $n_{\xi,i}\in\omega$, and $E_{\xi,i,j}$ is a basic open ball. Let $D_{\xi,i}$ be the set
	\[
		D_{\xi,i} := \bigcup_{j\leq n_{\xi,i}} E_{\xi,i,j}.
	\]	
	Our finitary name must satisfy:
	\begin{enumerate}
		\item $r(E_{\xi,i,j}) \leq 2^{-s-1}$;
		
		\item $D_{\xi,i}$ and $D_{\xi,j}$ are formally non-intersecting for $i\neq j$;
		
		\item $D_{\xi,i}$ is formally included into $U^{\mathcal{M}}_{\xi}$.
	\end{enumerate}	
	If such a name is found, then we put 
	$U^{\mathcal{M}}_{\xi\widehat{\ }i} := D_{\xi,i}$,
	and proceed to the next stage.
	
	\ 
	
	This concludes the description of the construction. In order to finish the proof, we need to establish that every stage successfully finds its own appropriate finitary name. 

	Consider stage $s+1$. We have a finite collection of clopen sets $U_{\xi}^{\mathcal{M}}$, where $\xi \in T_{\mathcal{M}}$ and $|\xi| = s$. 
	Since the space $\widehat{\mathcal{B}}$ is totally disconnected, there exist disjoint non-empty clopen sets $V_{\xi}$ and $W_{\xi}$ such that $V_{\xi} \cup W_{\xi} =  U_{\xi}^{\mathcal{M}}$. In what follows, we always assume that $|\xi| = s$.

We define (in a non-effective way) a countable cover $\mathcal{E}_{\xi}$ of the set $V_{\xi}$. The cover $\mathcal{E}_{\xi}$ contains all open balls $B(\alpha_i, r)$ with rational $r$ such that:
	\begin{enumerate}
		\item $r \leq 2^{-s-1}$;
		
		\item $B(\alpha_i,r)$ is formally included into some $C$, where $C$ is a basic open ball taken from the finitary name of $U_{\xi}^{\mathcal{M}}$;
		
		\item $B(\alpha_i, \frac{3}{2}r) \subseteq V_{\xi}$.
	\end{enumerate}
	The cover $\mathcal{E}_{\xi}$ has a finite subcover $\mathcal{E}'_{\xi} = \{ B(\alpha_{i_0},r_0), B(\alpha_{i_1}, r_1), \dots, B(\alpha_{i_{m}}, r_{m})\}$. Notice that $B(\alpha_{i_j}, \frac{3}{2} r_j) \subseteq V_{\xi}$ for all $j\leq m$.
	
	We consider an open cover $\mathcal{F}_{\xi}$ of the set $W_{\xi}$. This cover contains all balls $B(\alpha_i, r)$ with rational $r$ such that:
	\begin{enumerate}
		\item $r \leq 2^{-s-1}$;
		
		\item $B(\alpha_i,r)$ is formally included into some basic open $C$, taken from the finitary name of $U_{\xi}^{\mathcal{M}}$;
		
		\item $B(\alpha_i,r) \subseteq W_{\xi}$, and 
		
		\item for every $j\leq m$, we have $d(\alpha_{i_j},\alpha_i) > r + \frac{4}{3} r_j$.
	\end{enumerate}
	The cover $\mathcal{F}_{\xi}$ has a finite subcover $\mathcal{F}'_{\xi}$. 
	
	It is not hard to show that by combining the covers $\mathcal{E}'_{\xi} \cup \mathcal{F}'_{\xi}$, for all $\xi$, one can obtain a finitary name that we were looking for (at the stage $s+1$). This concludes the proof of Lemma~\ref{lem:aux-tree-splitting}. 
\end{proof}

	For a string $\sigma \in 2^{<\omega}$, we set $\ell(\sigma) := \sigma\widehat{\ }0$ and $r(\sigma) := \sigma\widehat{\ }1$.

	Let $T \subset \omega^{<\omega}$ be a finitely branching tree with a computable branching function such that every $\sigma\in T$ satisfies $b_T({\sigma}) \geq 2$. We define a computable function $\psi_T \colon T\to 2^{<\omega}$ as follows.
	\begin{itemize}
		\item[(a)] $\psi_T(\emptyset) = \emptyset$.
		
		\item[(b)] Suppose that $\psi_T(\sigma)$ is already defined. Using the branching function $b_{T}$, we find all children $\tau_0,\tau_1,\dots,\tau_k$ of $\sigma$ inside $T$. We put $\psi_T(\tau_0) = \ell(\psi_T(\sigma))$, $\psi_T(\tau_1) = \ell (r (\psi_T(\sigma)))$, $\psi_T(\tau_2) = \ell r r (\psi_T(\sigma))$, \dots, $\psi_T(\tau_{k-1}) = \ell r^{k-1}(\psi_T(\sigma))$, and $\psi_T(\tau_k) = r^k (\psi_T(\sigma))$.
	\end{itemize}
	One can show that the function $\psi_T$ induces a bijection from the set $[T]$ onto the set of all paths through the full binary tree.
	
	\ 
	
	Let $\mathcal{M}$ be a compact presentation of the space $\widehat{\mathcal{B}}$. Let $\mathcal{M}_{st}$ be a standard compact presentation of the Cantor space $2^{\omega}$. We will build an effectively continuous surjective homeomorphism $f$ acting from $\mathcal{M}$ onto $\mathcal{M}_{st}$. 
	
	By Remark~\ref{remark:Brattka}, this is enough for our purposes. Indeed, if $\mathcal{M}_0$ and $\mathcal{M}_1$ are two compact presentations of $\widehat{\mathcal{B}}$, then our construction shows the existence of effectively continuous $f_0\colon \mathcal{M}_0 \to \mathcal{M}_{st}$ and $f_1\colon \mathcal{M}_1 \to \mathcal{M}_{st}$.
	Then the map $f_1^{-1}\circ f_0$ is an effectively continuous surjective homeomorphism from $\mathcal{M}_0$ onto $\mathcal{M}_1$.
	
	Given $\mathcal{M}$, we use Lemma~\ref{lem:aux-tree-splitting} and recover computable tree $T_{\mathcal{M}}$ and sequence $\{ U^{\mathcal{M}}_{\sigma} \,\colon \sigma \in T_{\mathcal{M}}\}$ of clopen sets. 
	
	Our surjective homeomorphism $f$ is built as follows. For a point $x\in \mathcal{M}$, there is a unique path $P$ through $T_{\mathcal{M}}$ such that $\{ x \} = \bigcap_{\sigma \in P} U^{\mathcal{M}}_{\sigma}$. Using the map $\psi_{T_{\mathcal{M}}}$ discussed above, we transform the path $P$ into a path $P^{\ast}$ through the full binary tree. This path is an element of the Cantor space, and we put $f(x) := P^{\ast}$.
	
	We prove that the homeomorphism $f$ is effectively continuous. By Lemma~\ref{lem:effective-continuous}, it is sufficient to construct an enumeration operator $\Psi$, which given the name $N^x$ outputs the name of the point $f(x)$.
	
	The operator $\Psi$ acts as follows. Whenever its input data provides a basic open ball $B(c,r)$ (in $\mathcal{M}$) such that it is formally included into some $U^{\mathcal{M}}_{\sigma}$ for $\sigma\in T_{\mathcal{M}}$, $\Psi$ starts outputting all basic open balls $C$ (in $\mathcal{M}_{st}$) such that $C\supseteq B(\psi_{T_{\mathcal{M}}}(\sigma), 2^{-|\psi_{T_{\mathcal{M}}}(\sigma)|-1})$.
	
	\ 
	
	The theorem for the case of atomless $\mathcal{B}$ is proved. Now suppose that $\mathcal{B}$ has precisely $N\geq 1$ atoms. For simplicity, we discuss the case when $N = 1$. 
	
	The space $\widehat{\mathcal{B}}$ has a unique isolated point. Consider a compact presentation $\mathcal{M}$ of $\widehat{\mathcal{B}}$. Without loss of generality, we assume that the isolated point is equal to $\alpha_0$. We fix a (small enough) rational $R$ such that $B(\alpha_0, R) = \{ \alpha_0\}$.
	
	We employ the construction of Lemma~\ref{lem:aux-tree-splitting} with the following key modification: we require that every set $U^{\mathcal{M}}_{\sigma}$ formally does not intersect with $B(\alpha_0,R)$. Then the construction produces a tree $T_{\mathcal{M}}$ with the same nice properties.
	
	We fix a computable tree 
	\[
		T_{st} = \{ 0^n\,\colon n\in\omega\} \cup \{ 1\widehat{\ }\sigma \,\colon \sigma \in 2^{<\omega}\}.
	\]
	Clearly, the set $[T_{st}]$ (treated as a subspace of $2^{\omega}$) is homeomorphic to $\widehat{\mathcal{B}}$. Let $\mathcal{M}_{st}$ be the standard compact presentation of $[T_{st}]$.
	
	A surjective homeomorphism $f$ from $\mathcal{M}$ onto $\mathcal{M}_{st}$ acts as follows:
	\begin{itemize}
		\item[(a)] The point $\alpha_0$ is mapped to the unique isolated path of $T_{st}$.
		
		\item[(b)] Any other point $x$ corresponds to a path $P$ through $T_\mathcal{M}$. We recover the corresponding path $P^{\ast}$ through the full binary tree, and set $f(x) := 1\widehat{\ } P^{\ast}$ (this is a path through $T_{st}$).
	\end{itemize}
	
	The corresponding enumeration operator $\Psi$ is recovered similarly to the atomless case, modulo the following: if the input data provides an open ball $B(\alpha_0,r)$ (in $\mathcal{M}$) for some $r\leq R$, we start outputting the balls $B(0^{k+1}, 2^{-k-2})$ (in $\mathcal{M}_{st}$) for $k\in\omega$.
	
	This concludes the proof of Theorem~\ref{theo:comp-cat}.
\end{proof}

\ 

Note that the proof of~(b)$\Rightarrow$(a) given above is fully relativizable.

\section{Banach-Stone theorem for Stone spaces. Proof of Theorem~\ref{thm1}} Let $\check{X}$ denote the Boolean algebra dual to a Stone space $X$.
In \cite{HTMN-ta}, it was shown that $(2)$ of Theorem~\ref{thm1} is equivalent to:
\begin{itemize}
\item[(3)]  The Boolean algebra $\check{X}$ has a computable presentation.
\end{itemize}
 We therefore prove  $(1) \iff (3)$.
To see why $(3) \implies (1)$, represent the Stone space $X$ of $\check{X}$ as a computable subset of the Cantor set in $[0,1]$.
This is done by associating elements of the Boolean algebra with clopen sets in a closed subspace of the Cantor space; see the proof of the previous theorem.
For any clopen $Z$, let $f_Z$ be the function equal to $1$ at $Z$ and to $0$ at $X \setminus Z$.
Note that $f_Z$  is continuous, and that the linear $\mathbb{Q}$-span  of $\{f_Z \colon Z \subset_{clopen} X\}$ is dense in $C(X;\mathbb{R})$.
It remains to note that the distances between $qf_Z$ and $r f_Y$ are uniformly decidable, for any $r, q \in \mathbb{Q}$ and any clopen $Z,Y$.
The rest of the proof is devoted to checking  $(1) \implies (3)$.

\

\noindent \emph{Proof idea for $(1) \implies (3)$.} 
Suppose  $C(X;\mathbb{R})$ has a computable presentation as a Banach space; that is, the norm, $+$, and scalar multiplication are represented by uniformly computable operators. It is sufficient to assume that $+$ and the norm are computable. Equivalently, we can assume that the point $0$, the associated metric $d$, and $+$ are computable; for the details see, e.g.,  Fact 2.10 of \cite{MelNg}.
Recall that every $\Delta^0_2$ Boolean algebra in which the atom relation is $\Delta^0_2$ admits a computable  presentation;
this was essentially proven  in \cite{DoJo} but appears in this exact form in \cite{knight2000}. Thus, it is sufficient to build a $\Delta^0_2$ copy of the Boolean algebra of clopen sets in $X$, such that the atoms are $\Delta^0_2$.

	\medskip
	
\

\noindent \emph{Proof of $(1) \implies (3)$.} 	Fix a computable metric space $M$ whose completion $\overline{M}$ is isomorphic to $C(X;\mathbb{R})$ and where the operation of addition and the point $0$ are computable. Think of $M$ as a particular subset of $C(X;\mathbb{R})$ via a particular isomorphism between $\overline{M}$ and $C(X;\mathbb{R})$. Non-uniformly fix $p_0,p_1 \in M$ taking values in the interval $(0,1)$ and with $d(p_0,0) < \frac{1}{32}$ and $d(p_1,1) < \frac{1}{32}$, where $1$ is the constant function $1(x) =1$, and $0$ is the zero of $C(X ; \mathbb{R})$. Without loss of generality, we can assume $p_0 = 0$ by adding the computable point to the dense computable set of the space if necessary. We begin with a number of definitions and claims about these definitions.
	
	\medskip
	
	\begin{definition}
		Given $f \in C(X;\mathbb{R})$ with $-\frac{1}{32} \leq f \leq 1+\frac{1}{32}$ and for each $x \in X$, $f(x) < \frac{1}{4}$ or $f(x) > \frac{3}{4}$, we identify $f$ with the clopen set
		\[ X_f := \left\{ x : f(x) > \frac{1}{2}\right\} = \left\{ x : f(x) \geq \frac{1}{2}\right\}.\]
		We say that such an $f$ is an \textit{indicator function}.
	\end{definition}

	These might best be thought of as \textit{approximate} indicator functions, though for simplicity we will drop the term approximate. The idea is that they approximate the \textit{exact} indicator function $1_Y$ of a clopen set $Y$, where
	\[ 1_Y(x) = \begin{cases}
	
	0 & x \notin Y \\ 1 & x \in Y
	\end{cases}\]
	The main idea of the proof is to represent clopen subsets of $X$ by a corresponding indicator function in $M$. It is easy, for example, to see when two indicator functions represent the same clopen set.

	\begin{remark}\label{rem:eq}
		If $f,g \in C(X;\mathbb{R})$ are indicator functions, and $d(f,g) \leq \frac{1}{4}$, then $X_f = X_g$.
	\end{remark}
	
	Note that the exact indicator function $1_Y$ of a clopen set may not be in $M$, but there will always be an approximate indicator function in $M$, namely a close enough approximation to $1_Y$.
	
	\begin{remark}\label{rem:exist}
		Suppose that $Y$ is a clopen set and $\epsilon > 0$. Then there is an indicator function $f \in M$ such that $X_f = Y$ and $d(1_Y,f) < \epsilon$.
	\end{remark}

	To build the Boolean algebra of clopen sets, we will want to split a clopen set $Y$ into a disjoint union of two clopen sets, and then split those, and so on. So we need to see how this corresponds to indicator functions. The easiest case is when two indicator functions split  the whole space $X$. Recall that $p_1$ is a fixed approximation to the constant function~1.

	\begin{definition}\label{def:2-part}
		Given $f,g \in M$, we say that $f$ and $g$ form a 2-partition if:
		\begin{enumerate}
			\item $d(0,f) \leq 1$ and $d(0,g) \leq 1$;
			\item $d(p_1,f) \leq 1$ and $d(p_1,g) \leq 1$;
			\item $d(p_1,f+g) \leq \frac{1}{32}$;
			\item for all $q \in C(X;\mathbb{R})$ with $d(0,q) > \frac{1}{64}$, at least one of the distances $d(0,f+q)$, $d(0,f-q)$, $d(0,g+q)$, or $d(0,g-q)$ is $\geq 1 + \frac{1}{128}$.
		\end{enumerate}
	\end{definition}
	
	Note that this property is $\Pi^0_1$; it suffices to check (4) for $q \in M$.

	\begin{lemma}\label{claim:2-partition}
		Suppose that $f,g$ is a 2-partition. Then $f$ and $g$ are indicator functions, and $X$ is a disjoint union
		\[ X = X_{f} \sqcup X_{g}.\]
	\end{lemma}
	\begin{proof}
		By (1) and (2), for each $x \in X$, $-\frac{1}{32} \leq f(x) \leq 1+\frac{1}{32}$ and $-\frac{1}{32} \leq g(x) \leq 1+\frac{1}{32}$.
		
		Using (4), we argue that for each $x$ either $f(x) > 1 - \frac{1}{8}$ or $g(x) > 1 - \frac{1}{8}$. 
		Suppose instead that for some $x$, $f(x),g(x) \leq 1-\frac{1}{8}$. Then there is a clopen set $U \ni x$ such that for all $y \in U$ we have $f(y),g(y) < 1-\frac{1}{16}$. Define $q(y) = \frac{1}{32}$ for all $y \in U$, and $q(y) = 0$ for all $y \notin U$. Then all of the distances $d(0,f+q)$, $d(0,f-q)$, $d(0,g+q)$, and $d(0,g-q)$ are $< 1 + \frac{1}{128}$, contradicting (4). So we conclude that for each $x$ either $f(x) > 1 - \frac{1}{8}$ or $g(x) > 1 - \frac{1}{8}$.

		Now we argue that for each $x$, it is not the case that both $f(x) > \frac{1}{4}$ and $g(x) > \frac{1}{4}$. Indeed, if this was the case, then as either $f(x) > 1 - \frac{1}{8}$ or $g(x) > 1 - \frac{1}{8}$, we would have $f(x) + g(x) > 1 - \frac{1}{8} + \frac{1}{4} = 1 + \frac{1}{8}$, contradicting $d(p_1,f+g) \leq \frac{1}{32}$ as in (3).
		So we have shown that $f$ and $g$ are both indicator functions, and $X = X_{f} \sqcup X_{g}$. 
		$\qedhere$
	\end{proof}

	With more functions, there is something a little more complicated going on. Suppose that we tried to define when $f_1,\ldots,f_n \in M$ form an $n$-split by taking Definition \ref{def:2-part} and replacing ($3$) by ($3'$): $d(p_1,f_1+\cdots+f_n) \leq \frac{1}{32}$. We'd want to show that $X = X_{f_1} \sqcup \cdots \sqcup X_{f_n}$. The problem is that, for example, for a given $x$ and when $n$ is large, even though each $f_i(x)$ might be very small, the sum $f_1(x) + \cdots + f_n(x)$ might be large. (There are other similar ways problems could arise, such as with many $f_i(x)$ being negative.)  One way to get around this is to think of any splitting of the $n$-partition into two halves to form a 2-partition; see (1) in the definition below.

	\begin{definition}\label{def:refine}
		Given indicator functions $f_1,\ldots,f_n \in M$ and $g',g'' \in M$, we say that $h_1',h_1'',\ldots,h_n',h_n'' \in M$ is a \emph{partition of unity refining} $f_1,\ldots,f_n$ \emph{by} $g',g''$ if:
		\begin{enumerate}
			\item for every splitting $\Lambda \sqcup \Gamma = \{h_1',h_1'',\ldots,h_n',h_n''\}$, the functions $\sum_{t \in \Lambda} t$ and $\sum_{t \in \Gamma} t$ form a 2-partition;
			\item $d(f_i,h_i'+h_i'') \leq \frac{1}{4}$;
			\item $d(g',h_1'+\cdots+h_n') \leq \frac{1}{4}$ and $d(g'',h_1''+\cdots+h_n'') \leq \frac{1}{4}$.
		\end{enumerate}
	\end{definition}
	Since being a 2-partition is a $\Pi^0_1$ property, this property is also $\Pi^0_1$.

	\begin{remark}\label{rem:empty} Note that it is consistent with the definition above that some of the $h'_i$ or $h_j''$ can be indicator functions of the empty set.  Nonetheless, we can decide whether an indicator function $\xi$ represents $\emptyset$ by checking  if $d(0, \xi) > \dfrac{3}{4}$ or $d(0, \xi)<  \dfrac{1}{4}$; note that these properties are exclusive for an indicator function.
	\end{remark}
	
	\begin{lemma}
		Suppose that $h_1',h_1'',\ldots,h_n',h_n'' \in M$ is a partition of unity refining $f_1,\ldots,f_n$ by $g',g''$. Then each $h_i'$ and $h_i''$ is an indicator function and $X$ is a disjoint union
		\[ X = X_{h_1'} \sqcup X_{h_1''} \sqcup \cdots \sqcup  X_{h_n'} \sqcup X_{h_n''}.\]
		We have
		\[ X_{g'} = X_{h_1'} \sqcup \cdots \sqcup  X_{h_n'},\]
		\[ X_{g''} = X_{h_1''} \sqcup \cdots \sqcup X_{h_n''}, \]
		and, for each $i$,
		\[ X_{f_i} = X_{h_i'} \sqcup X_{h_i''}.\]
	\end{lemma}
	\begin{proof}
		For a given $\Lambda \subseteq \{h_1',h_1'',\ldots,h_n',h_n''\}$, denote by $h_\Lambda$ the function $\sum_{t \in \Lambda} t$. By (1) and Lemma \ref{claim:2-partition}, for every $\Lambda \subseteq \{h_1',h_1'',\ldots,h_n',h_n''\}$, $h_{\Lambda}$ is an indicator function. Moreover, for every splitting $\Lambda \sqcup \Gamma = \{h_1',h_1'',\ldots,h_n',h_n''\}$, we have a partition
		\[ X = X_{h_\Lambda} \sqcup X_{h_\Gamma}.\]
		In particular, the $h_i'$ and $h_{i}''$ are indicator functions as, e.g., $h_i' = h_\Lambda$ for $\Lambda = \{h_i'\}$.
		
		Now we want to argue that for each $\Lambda \subseteq \{h_1',h_1'',\ldots,h_n',h_n''\}$, we have a disjoint union
		\[ X_{h_{\Lambda}} = \bigsqcup_{t \in \Lambda} X_t. \]
		It suffices to show that given $\Lambda$ and $t \in \Lambda$, there is a disjoint union
		\[ X_{h_\Lambda} = X_{h_{\Lambda - \{t\}}} \sqcup X_t.\]
		Since $h_\Lambda$, $h_{\Lambda - \{t\}}$, and $h_t$ are all indicator functions, and $h_\Lambda = h_{\Lambda - \{t\}} + h_t$:
		\begin{itemize}
			\item For a given $x \in X$ it cannot be that both $h_{\Lambda - \{t\}}(x) > \frac{3}{4}$ and $h_t(x) > \frac{3}{4}$, as then we would have $h_\Lambda(x) > \frac{3}{2}$. Thus $X_{h_{\Lambda - \{t\}}}$ and $X_t$ are disjoint.
			\item For a given $x \in X$, if $h_{\Lambda - \{t\}}(x) > \frac{3}{4}$ then as $h_t(x) \geq - \frac{1}{32}$, $h_\Lambda(x) > \frac{1}{2}$. Thus $X_{h_{\Lambda - \{t\}}} \subseteq X_{h_\Lambda}$. Similarly, $X_t \subseteq X_{h_\Lambda}$.
			\item For a given $x \in X$, if $h_{\Lambda - \{t\}}(x) < \frac{1}{4}$ and $h_t(x) < \frac{1}{4}$, then $h_\Lambda(x) < \frac{1}{2}$. Thus $X_{h_\Lambda} \subseteq X_{h_{\Lambda - \{t\}}} \cup X_t$.
		\end{itemize}
		Putting this all together, we see that
		\[ X_{h_\Lambda} = X_{h_{\Lambda - \{t\}}} \sqcup X_t.\]
		By repeated applications, we conclude that for each $\Lambda \subseteq \{h_1',h_1'',\ldots,h_n',h_n''\}$,
		\[ X_{h_{\Lambda}} = \bigsqcup_{t \in \Lambda} X_t. \tag{$*$}\]
		
		Applying ($*$) to any splitting $X = X_{h_\Lambda} \sqcup X_{h_\Gamma}$ we immediately get that
		\[ X = X_{h_1'} \sqcup X_{h_1''} \sqcup \cdots \sqcup  X_{h_n'} \sqcup X_{h_n''}.\]
		
		For each $i$, by (2) we have $d(f_i,h_i'+h_i'') \leq \frac{1}{4}$. Since $f_i$ and $h_i'+h_i''$ are both indicator functions, this implies easily that
		\[ X_{f_i} = X_{h_i' + h_i''} = X_{h_i'} \sqcup X_{h_i''}.\]
		For the first equality we use Remark \ref{rem:eq}, and for the second we use ($*$).
		
		By (3), we have that $d(g',h_1'+\cdots+h_n') \leq \frac{1}{4}$. Using this and ($*$), we get that
		\[ X_{g'} = X_{h_1'+\cdots+h_n'} = X_{h_1'} \sqcup \cdots \sqcup  X_{h_n'},\]
		Similarly,
		\[ X_{g''} = X_{h_1''+\cdots+h_n''} = X_{h_1''} \sqcup \cdots \sqcup X_{h_n''}.\qedhere \]	
	\end{proof}

\medskip

Recall that our goal is to build a $\Delta^0_2$ copy of the Boolean algebra of clopen sets in $X$, such that the atoms are $\Delta^0_2$. For this we will need a way to tell in a $\Delta^0_2$ way whether an indicator function $f$ represents an atom $X_f$. The definitions above are a little too complicated for this; the problem is that they involve searching for indicator functions. In the following definition, the $g$ and $h$ need not be indicator functions, so they witness that $f$ does split without finding an actual splitting.

\begin{definition}
	Given $f \in M$, we say that $f$ \textit{splits} if there are $g,h \in C(X;\mathbb{R})$ such that:
	\begin{enumerate}
		\item $d(p_1,g) < 1$ and $d(p_1,h) < 1$;
		\item $\frac{3}{4} < d(0,g) < 1$ and $\frac{3}{4} < d(0,h) < 1$;
		\item $d(f,g+h) < \frac{1}{32}$.
	\end{enumerate}
\end{definition}
This is $\Sigma^0_1$; again, it suffices to check for $f,g \in M$.
We will apply the above notion only in case when $f$ is an indicator function of a non-empty set; see Remark~\ref{rem:empty}. (If $f$ represents $\emptyset$ then it actually does not split.)

\begin{lemma}
	Suppose that $f$ is an indicator function and that $f$ splits. Then $X_f$ is not an atom.
\end{lemma}
\begin{proof}
	By (1), we have that $d(p_1,g) < 1$ and $d(p_1,h) < 1$. By (2), we have that $d(0,g) < 1$ and $d(0,h) < 1$. Thus, for all $x$, $-\frac{1}{32} < g(x) < 1+\frac{1}{32}$ and $-\frac{1}{32} < h(x) < 1+\frac{1}{32}$.
	
	Since $\frac{3}{4} < d(0,g)$, we can choose $x$ with $g(x) > \frac{5}{8}$. Then by (3),
	\[ f(x) > g(x)+h(x) - \frac{1}{32} > \frac{1}{2}.\]
	So $x \in X_f$. Similarly, we can choose $y$ with $h(y) > \frac{5}{8}$, and $y \in X_f$.
	
	Finally, we claim that $x \neq y$. Indeed, if $x = y$, then $g(x) +h(x) > \frac{5}{4}$, and $f(x) < 1+\frac{1}{32}$, contradicting (3). So $X_f$ contains at least two distinct elements. \end{proof}

	\medskip
	
	We are now ready for the construction. 
	We are building a $\Delta^0_2$ copy of the Boolean algebra of clopen sets in $X$, with the atom relation being $\Delta^0_2$ as well.  Denote by $\mc{B}(X)$ this Boolean algebra of clopen sets. The construction will use a $\mathbf{0}'$ oracle, and hence can see whether elements form 2-partitions, partitions of unity, or split. At each stage, we will have a finite subalgebra $\mc{B}_s$, extended at each subsequent stage, such that $\mc{B} = \bigcup_s \mc{B}_s$ is isomorphic to the algebra of clopen sets in $X$. At each stage $s$, $\mc{B}_s$ will be a finite Boolean algebra whose atoms are all indicator functions from $M$, and the elements of $\mc{B}_s$ can be thought of as formal terms in these indicator functions. At every stage $s$, we will have an embedding $\varphi_s \colon \mc{B}_s \to \mc{B}(X)$ induced by mapping an atom $f \in M$ of $\mc{B}_s$ to $X_{f}$. The union $\varphi = \bigcup_s \varphi_s \colon \mc{B} \to \mc{B}(X)$ will be an isomorphism. Some of the atoms of $\mc{B}$ will be labeled with the atom relation $\At$, signifying that they will be atoms of $\mc{B}$; those not labeled $\At$ will later be split.
	
	Let $(q_s,r_s)_{s \in \omega}$ be a listing of all pairs of elements $q_s,r_s \in M$. The idea is that at stage $s+1$, if $q_s$ and $r_s$ are indicator functions splitting the domain $X$, then the atoms of $\mc{B}_{s+1}$ will induce a refinement of this splitting.
	
	\medskip
	
\noindent \emph{Construction.} 	 Begin with  the Boolean algebra $\mc{B}_0 = \{0, p_1\}$ generated by  $p_1$.
	
	\medskip
	
	At stage $s+1$, suppose that $\mc{B}_s$ is the Boolean algebra with atoms $f_1,\ldots,f_n,g_1,\ldots,g_m$, with $\At(f_1),\ldots,\At(f_n)$ and $\neg \At(g_1),\ldots,\neg\At(g_m)$. First, ask whether $q_s,r_s$ is a 2-partition. If not, then set $\mc{B}_{s+1} = \mc{B}_s$ and end this stage of the construction. If $q_s,r_s$ is a 2-partition, look for $f_1',f_1'',\ldots,f_n',f_n'' \in M$ and $g_1',g_1'',\ldots,g_m',g_m'' \in M$ such that:
	\begin{itemize}
		\item $f_1',f_1'',\ldots,f_n',f_n'',g_1',g_1'',\ldots,g_m',g_m''$ is a partition of unity refining $f_1,\ldots,f_n,g_1,\ldots,g_m$ by $q_s,r_s$.
	\end{itemize}

	By Remark~\ref{rem:empty}, we can check which indicator functions represent the empty set; we call such an indicator function trivial. For instance, since $f_1,\ldots,f_n$ are the indicator functions of atoms, for each $i = 1,\ldots,n$, one of $f_i'$ or $f_i''$ is an indicator function for the empty set, and the other is an indicator function for $X_{f_i}$. 	Then let $\mc{B}_{s+1}$ be the Boolean algebra generated by $f_1, \ldots, f_n, g_1',g_1'',\ldots,g_m',g_m''$ where the trivial indicator functions are set equal to $0$. Embed $\mc{B}_{s}$ into $\mc{B}_{s+1}$ by mapping $f_i \mapsto f_i$ and $g_i \mapsto g_i' \vee g_i''$. Put $\At(f_1),\ldots,\At(f_n)$. For each $i$, put $\At(g_i')$ if $g_i'$ does not split, and similarly for $g_i''$.

	\
	
	\noindent \emph{Verification.} As we have already explained above, $\mathbf{0}'$ is sufficient to arrange the construction, in the sense that every property that we need to check is $\Delta^0_2$.
	We begin with the lemma below which says that every stage will eventually successfully finish its search.\
	\begin{lemma}
		If $q_s,r_s$ form a 2-partition, then there exist $f_1',f_1'',\ldots,f_n',f_n'' \in M$ and $g_1',g_1'',\ldots,g_m',g_m'' \in M$ such that:
		\begin{itemize}
			\item $f_1',f_1'',\ldots,f_n',f_n'',g_1',g_1'',\ldots,g_m',g_m''$ is a partition of unity refining $f_1,\ldots,f_n,g_1,\ldots,g_m$ by $q_s,r_s$.
		\end{itemize}
		(Here, $f_1,\ldots,f_n,g_1,\ldots,g_m$ should be thought of as being the atoms of $\mc{B}_{s}$.  We intend to define $\mc{B}_{s+1}$ to be the  Boolean algebra with atoms $f_1,\ldots,f_n,g_1',g_1'',\ldots,g_m',g_m''$.)
		
	\end{lemma}
	\begin{proof}
		Since $q_s,r_s$ form a 2-partition, $q_s$ and $r_s$ are indicator functions and $X = X_{q_s} \sqcup X_{r_s}$. Let $\epsilon = \frac{1}{512}$.
		
		For each $i$, let $f_i' \in M$ be at distance at most $\epsilon/(8n+8m)$ from
		\[ x \mapsto  
		 \begin{cases}
		\epsilon/(4n+4m) & x \notin X_{q_s} \cap X_{f_i}  \\ 1 - 2 \epsilon & x \in X_{q_s} \cap X_{f_i} \end{cases}\]
		and let $f_i'' \in M$ be at distance at most $\epsilon/(8n+8m)$ from
		\[ x \mapsto  \begin{cases}
		\epsilon/(4n+4m) & x \notin X_{r_s} \cap X_{f_i}  \\ 1 - 2 \epsilon & x \in X_{r_s} \cap X_{f_i}  \end{cases}.\]
		We have that $X_{f_i'} = X_{q_s} \cap X_{f_i}$ and $X_{f_i''} = X_{r_s} \cap X_{f_i}$. Since $X_{f_i}$ is a singleton set, either $X_{f_i'} = \varnothing$ and $X_{f_i''} = X_{f_i}$, or $X_{f_i'} = X_{f_i}$ and $X_{f_i''} = \varnothing$.
		
		Similarly, for each $i$, let $g_i' \in M$ be at distance at most $\epsilon/(8n+8m)$ from
		\[ x \mapsto  \begin{cases} 
		\epsilon/(4n+4m) & x \notin X_{q_s} \cap X_{g_i}\\ 1 - 2 \epsilon & x \in X_{q_s}\cap X_{g_i} \end{cases}\]
		and let $g_i'' \in M$ be at distance at most $\epsilon/(8n+8m)$ from
		\[ x \mapsto   \begin{cases}
		\epsilon/(4n+4m) & x \notin X_{r_s}\cap X_{g_i} \\ 1 - 2 \epsilon & x \in X_{r_s}\cap X_{g_i} \end{cases}.\]
		We have that $X_{g_i'} = X_{q_s} \cap X_{g_i}$ and $X_{g_i''} = X_{r_s} \cap X_{g_i}$.
		
		Note that \[X_{f_i'} \cup X_{f_i''} = X_{f_i}\] and
		\[X_{g_i'} \cup X_{g_i''} = X_{g_i},\]
		Also,
		\[ X_{q_s} = X_{f_1'} \cup \cdots \cup X_{f_n'} \cup X_{g_1'} \cup \cdots \cup X_{g_m'} \]
		and
		\[ X_{r_s} = X_{f_1''} \cup \cdots \cup X_{f_n''} \cup X_{g_1''} \cup \cdots \cup X_{g_m''}.\]
		
		From the following easy to prove claims, almost everything that we want will follow.
		
		\begin{claim}
			Let $\Lambda \subseteq \{f_1',f_1'',\ldots,f_n',f_n'',g_1',g_1'',\ldots,g_m',g_m''\}$. For each $x$:
			\begin{itemize}
				\item If $x \in X_t$ for some $t \in \Lambda$, then
				\[ 1-3\epsilon < \sum_{\xi \in \Lambda} \xi < 1-\epsilon.\]
				\item If $x \notin X_t$ for any $t \in \Lambda$, then
				\[ 0 < \sum_{\xi \in \Lambda} \xi < \epsilon.\]
			\end{itemize}
		\end{claim}
	
		\begin{claim}
			Let $\Lambda \subseteq \{f_1',f_1'',\ldots,f_n',f_n'',g_1',g_1'',\ldots,g_m',g_m''\}$. Let $h$ be an indicator function. If $X_h = \bigsqcup_{t \in \Lambda} X_t$, then $d(h,\sum_{t \in \Lambda} t) < \frac{1}{4}$.
		\end{claim}
		
		\medskip
		
		To prove the lemma, we must verify the following.
		\begin{enumerate}
			\item For every splitting $\Lambda \sqcup \Gamma = \{f_1',f_1'',\ldots,f_n',f_n'',g_1',g_1'',\ldots,g_m',g_m''\}$, the functions $\sum_{t \in \Lambda} t$ and $\sum_{t \in \Gamma} t$ form a 2-partition:
			
			Fix a splitting $\Lambda \sqcup \Gamma = \{f_1',f_1'',\ldots,f_n',f_n'',g_1',g_1'',\ldots,g_m',g_m''\}$. 
			
			\begin{enumerate}
				\item $d(0,\sum_{t \in \Lambda} t) \leq 1$ and $d(0,\sum_{t \in \Gamma} t) \leq 1$.		
				\item $d(p_1,\sum_{t \in \Lambda} t) \leq 1$ and $d(p_1,\sum_{t \in \Gamma} t) \leq 1$.
				
				\textit{Verification:} For both (a) and (b), note by the  claims that $\sum_{t \in \Lambda} t$ and $\sum_{t \in \Gamma} t$ take values in $[0,1]$. 
				
				\item $d(p_1,f_1'+f_1''+\cdots+f_n'+f_n''+g_1'+g_1''+\cdots+g_m'+g_m'') \leq \frac{1}{32}$.
				
				\textit{Verification:} By the first claim, for each $x$, the value of $f_1'+f_1''+\cdots+f_n'+f_n''+g_1'+g_1''+\cdots+g_m'+g_m''$ at $x$ is between $1-3\epsilon$ and $1-\epsilon$. The function $p_1$, by definition, takes values in the interval $(\frac{31}{32},1)$ and thus has distance at most $\frac{1}{32}$ from $f_1'+f_1''+\cdots+f_n'+f_n''+g_1'+g_1''+\cdots+g_m'+g_m''$. Thus (c) follows.
				
				\item for all $q \in C(X;\mathbb{R})$ with $d(0,q) > \frac{1}{64}$, one of the distances $d(0,\sum_{t \in \Lambda} t+q)$, $d(0,\sum_{t \in \Lambda} t-q)$, $d(0,\sum_{t \in \Gamma} t+q)$, or $d(0,\sum_{t \in \Gamma} t-q)$ is $\geq 1+\frac{1}{128}$.
				
				\smallskip
				
				\textit{Verification:} Given such a $q$, let $x$ be such that $q(x) > \frac{1}{64}$. By the claims, either $\sum_{t \in \Lambda} t(x) > 1-3\epsilon$ or $\sum_{t \in \Gamma} t(x) > 1-3\epsilon$. As  $3\epsilon < \frac{1}{128}$, (d) follows.
			\end{enumerate}
			
			\item $d(f_i,f_i'+f_i'') \leq \frac{1}{4}$ and $d(g_i,g_i'+g_i'') \leq \frac{1}{4}$.
	
			\item $d(q_s,f_1'+\cdots+f_n'+g_1'+\cdots+g_m') \leq \frac{1}{4}$ and $d(r_s,f_1''+\cdots+f_n''+g_1''+\cdots+g_m'') \leq \frac{1}{4}$.
			
			\textit{Verification:} (2) and (3) follow immediately from the two claims.
		\end{enumerate}
	Thus $f_1',f_1'',\ldots,f_n',f_n'',g_1',g_1'',\ldots,g_m',g_m''$ is a partition of unity refining $f_1,\ldots,f_n,g_1,\ldots,g_m$ by $q_s,r_s$.
	\end{proof}

	\begin{lemma}
		$\varphi \colon \mc{B} \to \mc{B}(X)$ is a surjective isomorphism of Boolean algebras.
	\end{lemma}
	\begin{proof}
		Given $U$ a clopen subset of $X$, using Remark \ref{rem:exist} let $q,r \in M$ be a 2-partition with $X_q = U$ and $X_r = X - U$. Let $s$ be such that $(q_s,r_s) = (q,r)$. Suppose that $\mc{B}_{s+1}$ is the Boolean algebra with atoms $f_1,\ldots,f_n,g_1',g_1'',\ldots,g_m',g_m''$. Then $U = X_q = \bigcup \{X_{f_i} : i \in I\} \cup X_{g_1'} \cup \cdots \cup X_{g_m'}$ for some $I\subseteq \{1,\dots,n\}$. So $\varphi(\bigvee_{i \in I}f_i \vee g_1' \vee \cdots \vee g_m') = U$.
	\end{proof}
	
To finish the proof of the theorem,  use the cited above theorem of Downey and Jockusch to produce  a computable presentation of the $\Delta^0_2$ Boolean algebra produced by the construction.

\bibliography{ourbib}
\bibliographystyle{alpha}

\end{document}